\documentclass[11pt]{amsart}
\usepackage{mathptmx, amssymb, amsmath, mathtools, enumerate, colonequals}

\usepackage{xcolor, color}
\usepackage[margin=1.2in]{geometry}
\definecolor{chianti}{rgb}{0.6,0,0}
\definecolor{meretale}{rgb}{0,0,.6}
\definecolor{leaf}{rgb}{0,.35,0}
\usepackage[colorlinks=true, pagebackref, hyperindex, citecolor=meretale, urlcolor=leaf, linkcolor=chianti]{hyperref}

\usepackage[mathscr]{euscript}

\DeclareFontFamily{OMS}{rsfs}{\skewchar\font'60}
\DeclareFontShape{OMS}{rsfs}{m}{n}{<-5>rsfs5 <5-7>rsfs7 <7->rsfs10 }{}
\DeclareSymbolFont{rsfs}{OMS}{rsfs}{m}{n}
\DeclareSymbolFontAlphabet{\scr}{rsfs}

\usepackage{amsmath,amsfonts,mathtools,amsthm, amssymb}
\usepackage{leftindex}
\usepackage{enumitem}
\usepackage{graphicx,float}

\usepackage[ruled,vlined]{algorithm2e}
\usepackage{algorithmic}

\usepackage{ytableau}
\usepackage{comment}
\usepackage{bigints}
\usepackage{setspace}
\numberwithin{equation}{section}

\newtheorem{theorem}{Theorem}[section]
\newtheorem{lemma}[theorem]{Lemma}
\newtheorem{prop}[theorem]{Proposition}
\newtheorem{corollary}[theorem]{Corollary}

\theoremstyle{definition}
\newtheorem{defn}[theorem]{Definition}
\newtheorem{example}[theorem]{Example}

\theoremstyle{remark}
\newtheorem{remark}[theorem]{Remark}

\newtheorem{notation}[theorem]{Notation}
\numberwithin{equation}{subsection}

\usepackage [english]{babel}
\usepackage [autostyle, english = american]{csquotes}
\MakeOuterQuote{"}

\newcommand{\fraka}{\mathfrak{a}}
\newcommand{\frakb}{\mathfrak{b}}
\newcommand{\frakm}{\mathfrak{m}}

\newcommand{\frakp}{\mathfrak{p}}

\newcommand{\NN}{\mathbb{N}}
\newcommand{\PP}{\mathbb{P}}
\newcommand{\ZZ}{\mathbb{Z}}

\newcommand{\calD}{\mathcal{D}}

\newcommand{\kk}{\Bbbk}

\newcommand{\Mod}{\operatorname{Mod}^*}
\newcommand{\hSpec}{\operatorname{Spec}^*}

\DeclareMathOperator{\Ext}{Ext}

\DeclareMathOperator{\Supp}{Supp}
\DeclareMathOperator{\Proj}{Proj}
\DeclareMathOperator{\Depth}{depth}
\DeclareMathOperator{\hgt}{height}
\DeclareMathOperator{\cd}{cd}

\DeclareMathOperator{\E}{E}
\DeclareMathOperator{\id}{id}
\DeclareMathOperator{\Char}{char}

\makeatletter
\@namedef{subjclassname@2020}{%
  \textup{2020} Mathematics Subject Classification}
\makeatother

\title{Local cohomology and Segr\'{e} products}
\author{Jiamin Li and Wenliang Zhang}
\address{Department of Mathematics, Statistics, and Computer Science, University of Illinois at Chicago,
Chicago, IL 60607}
\email{jli283@uic.edu, wlzhang@uic.edu}
\thanks{Both authors are partially supported by NSF through DMS-1752081.}
\subjclass[2020]{13D45, 14B15, 13A02}

\begin{document}

\maketitle

\begin{abstract}
We prove a K\"{u}nneth formula for local cohomology of a Segr\'{e} product of graded modules supported in a Segr\'{e} product of ideals. In order to apply our formula to the study of cohomological dimension, we also investigate asymptotic behaviors of Eulerian graded $\scr{D}$-modules.
\end{abstract}

\dedicatory{
\begin{center}\emph{Dedicated to Professor Kei-ichi Watanabe on the occasion of his 80th birthday.}\end{center}
}

%%%%%%%%%%%%%%%%%%%%%%%%%%%
\section{Introduction}
%%%%%%%%%%%%%%%%%%%%%%%%%%%
Let $\kk$ be a field and let $R=\kk[x_0,\dots,x_n]$ and $S=\kk[y_0,\dots,y_m]$ be polynomial rings over $\kk$. Then the homogeneous coordinate ring of the Segr\'{e} embedding $\PP^n\times \PP^m\hookrightarrow \PP^{mn+m+n}$ is given by the Segr\'{e} product of $R$ and $S$:
\[R\#S:=\bigoplus_{i\geq 0}R_i\otimes_{\kk}S_i.\]
More generally, given any two $\NN$-graded rings $R$ and $S$ such that $R_0=S_0=\kk$, the Segr\'{e} product $R\#S$ of $R$ and $S$ is defined as above. Segr\'{e} products of $\NN$-graded rings were considered in \cite{ChowUnmixedness} to investigate whether the Segr\'{e} product of Cohen-Macaulay $\NN$-graded rings is again Cohen-Macaulay. This approach was later extended in \cite[Theorem 4.1.5]{GotoWatanabeI} to a K\"{u}nneth type formula. 

\begin{theorem}[Goto-Watanabe]
\label{thm: GW}
Let $R,S$ be Noetherian $\NN$-graded rings with $R_0=S_0=\kk$ and let $M$ (and $N$) be a $\ZZ$-graded $R$-module (a $\ZZ$-graded $S$-module, respectively). Denote by $\frakm_R$, $\frakm_S$ and $\frakm_{R\#S}$ the homogeneous maximal ideals in $R$, $S$, and $R\#S$ respectively. Assume that $H^i_{\frakm_R}(M)=H^i_{\frakm_R}(N)=0$ when $i=0,1$. Then, for each $k\geq 2$,
\[
H^k_{\frakm_{R\#S}}(M\#N)\cong \Big( M\#H^k_{\frakm_S}(N)\Big)\oplus \Big(H^k_{\frakm_R}(M)\#N\Big)\oplus \Big(\bigoplus_{i+j=k+1}H^i_{\frakm_R}(M)\#H^j_{\frakm_S}(N)\Big)
\]
\end{theorem}

One of our main results in this article extends this K\"{u}nneth type formula to more general ideals as follows.
\begin{theorem}[=Theorem \ref{Kunneth formula: two components}]
\label{Kunneth formula: intro}
Let $\kk$ be a field. Let $R$ and $S$ be standard graded $\kk$-algebras. Let $I$ be a homogeneous ideal of $R$ and $J$ be a homogeneous ideal of $S$. Then for every $\ZZ$-graded $R$-module $M$ and every $\ZZ$-graded $S$-module $N$ we have an exact sequence
\[
0\to H^0_{I\#J}(M\#N)\to M\#N\to M^{sat}_I\#N^{sat}_J\to H^1_{I\#J}(M\#N)\to 0
\]
and isomorphisms
\[
H^k_{I\#J}(M\#N)\cong \Big( M^{sat}_I\#H^k_J(N)\Big)\oplus \Big(H^k_I(M)\#N^{sat}_J\Big)\oplus \Big(\bigoplus_{i+j=k+1}H^i_I(M)\#H^j_J(N)\Big)
\]
for all $k\geq 2$. 
\end{theorem}
In Theorem \ref{Kunneth formula: intro}, $M^{sat}_I$ (and $N^{sat}_J$) denotes the $I$-saturation of $M$ (and the $J$-saturation of $N$) which is defined in Definition \ref{defn: sat}.

Note that $\frakm_{R\#S}=\frakm_R\#\frakm_S$ in the settings of Theorem \ref{thm: GW}. Hence Theorem \ref{Kunneth formula: intro} can be viewed as a generalization of Theorem \ref{thm: GW}.

One of our motivations behind Theorem \ref{Kunneth formula: intro} is the study of cohomological dimension ({\it cf.} Definition \ref{cd defn}) of ideals in a non-regular ring. To this end, we also investigate Eulerian graded $\scr{D}$-modules; local cohomology modules of a polynomial ring supported in a homogeneous ideal are primary examples of Eulerian graded $\scr{D}$-modules ({\it cf.} \S\ref{graded D-module and cd section} for details). One of our main results on Eulerian graded $\scr{D}$-modules is the following:

\begin{theorem}[=Theorem \ref{nonzero in negative degree}]
\label{intro thm: nonzero in negative degree}
Let $R=\kk[x_1,\dots,x_d]$ be a polynomial ring over a field $\kk$  and let $M$ be a nonzero Eulerian graded $\scr{D}$-module. 
\begin{enumerate}
\item 
\label{supp dim 0}
If $\dim(\Supp_R(M))=0$, then 
\begin{enumerate}
\item $M_\ell=0$ for each integer $\ell>-d$, and
\item $M_{\ell}\neq 0$ for each integer $\ell\leq -d$.
\end{enumerate}
\item 
\label{supp dim p>0}
If $\dim(\Supp_R(M))>0$ and each element in $M$ is annihilated by a nonzero polynomial in $R$, then 
\[M_{\ell}\neq 0\quad \forall \ell\in \ZZ.\]
\end{enumerate}
In particular, if each element in $M$ is annihilated by a nonzero polynomial in $R$, then $M_{\ell}\neq 0$ for every integer $\ell\leq -d$.
\end{theorem}

Theorem \ref{intro thm: nonzero in negative degree} vastly generalizes \cite[Theorems 1.3 and 1.6]{PuthenpurakalGradedComponents} (which concern with local cohomology modules in equal-characteristic 0). 

Combining Theorems \ref{Kunneth formula: intro} and \ref{intro thm: nonzero in negative degree} also produces the following consequence on cohomological dimension:

\begin{theorem}[= Theorem \ref{thm: optimal in polynomial case}]
\label{intro thm: cd poly ring}
Let $R,S$ be polynomial rings over the same field $\kk$. Let $I,J$ be nonzero homogeneous ideals in $R,S$ respectively. Then
\[\cd_{I\# J}(R\# S)= \cd_I(R)+\cd_J(S)-1.\]
\end{theorem}

The article is organized as follows. In \S\ref{section: preliminaries}, we collect some materials on local cohomology and graded rings which are necessary for the subsequent sections. In \S\ref{Kunneth formula}, we prove Theorem \ref{Kunneth formula: intro}. \S\ref{depth section} contains an application to the depth of Segr\'{e} products of standard graded rings. Finally, we investigate asymptotic behaviors of Eulerian graded $\scr{D}$-modules and applications to cohomological dimension of Segr\'{e} products of polynomial rings in \S\ref{graded D-module and cd section}; in particular, this is where Theorems \ref{intro thm: nonzero in negative degree} and \ref{intro thm: cd poly ring} are proved.

%%%%%%%%%%%%%%%%%%%%%%%%%%%%
\section{Preliminaries on local cohomology and graded rings}
\label{section: preliminaries}
%%%%%%%%%%%%%%%%%%%%%%%%%%%%
To ease notations and technicalities, throughout this article we will focus on standard graded rings. Recall that an $\NN$-graded ring $R$ is a standard graded ring if $R\cong \kk[x_1,\dots,x_n]/\frakb$ where $\kk$ is a field and $\frakb$ is a homogeneous ideal in the polynomial ring $\kk[x_1,\dots,x_n]$ equipped with the standard grading: $\deg(x_i)=1$ for each $x_i$ and $\deg(c)=0$ for each $c\in \kk$. For each such $R$, we will denote by $\frakm$ the homogeneous maximal ideal (the ideal generated by the degree-$1$ piece $R_1$). We will denote the set of homogeneous prime ideals by $\hSpec(R)$ (equipped with the usual Zariski topology). Note that $\Proj(R)=\hSpec(R)\backslash\{\frakm\}$. We will denote by $\Mod_R$ the category of $\ZZ$-graded $R$-modules in which the objects are $\ZZ$-graded $R$-modules and the morphisms are graded $R$-module homomorphisms. We will denote by $M(a)$ the $\ZZ$-graded $R$-module such that $M(a)_i=M_{a+i}$ ($i\in \ZZ$) for each integer $a$ and each $M\in \Mod_R$.

We fix a field $\kk$; all graded rings considered in this article are standard graded rings over the same field $\kk$.

Let $I$ be a homogeneous ideal of $R$. The $I$-torsion functor $\Gamma_I:\Mod_R\to\Mod_R$ is defined as follows: 
\begin{enumerate}
\item $\Gamma_I(M)=\{z\in M\mid I^nz=0\ {\rm for\ some\ integer\ }n\}$ for each $M\in \Mod_R$;
\item $\Gamma_I(f)=f_{\Gamma_I(M)}:\Gamma_I(M)\to \Gamma_I(N)$ for each morphism $f:M\to N$ in $\Mod_R$.
\end{enumerate} 
One can check that $\Gamma_I$ is left-exact. Its $j$-th right derived functor $\mathscr{R}^j\Gamma_I$ is called the $j$-th local cohomology and is denoted by $H^j_I(-)$; that is, for each $M\in \Mod_R$, the $j$-th local cohomology of $M$ supported in $I$ is 
\[H^j_I(M)\cong H^j(\Gamma_I(E^{\bullet}))\]
where $E^{\bullet}$ is an injective resolution of $M$ in $\Mod_R$.

\begin{remark}
\label{structure of injectives}
We collect some basic facts on injective modules in $\Mod_R$ here ({\it cf.} \cite[\S13.2]{BrodmannSharp}).
\begin{enumerate}
\item Each indecomposable injective object in $\Mod_R$ has the form $\E(R/P)(\ell)$ for an integer $\ell$ where $P$ is a homogeneous prime ideal of $R$ and $\E(-)$ denotes the injective hull in $\Mod_R$. 

\item \label{structure of injective hull}
Each element in $\E(R/P)$ is annihilated by a power of $P$; if a homogeneous element $f$ is not in $P$, then
\[\E(R/P)(-\deg(f))\xrightarrow{f}\E(R/P)\]
is an isomorphism in $\Mod_R$.

\item An object $E\in \Mod_R$ is an injective object if and only if $E$ is a direct sum of some copies of $\E(R/P_{\alpha})(\ell_{\alpha})$ for homogeneous prime ideals $P_{\alpha}$ and integers $\ell_\alpha$. 

\item \label{I-torsion of injective}
Let $I$ be an homogeneous ideal of $R$, then it follows from (\ref{structure of injective hull}) that
\[\Gamma_I(\E(R/P)(\ell))=\begin{cases} \E(R/P)(\ell) & I\subseteq P\\ 0& {\rm otherwise}\end{cases}\]
\end{enumerate}
\end{remark}

This can be interpreted geometrically as follows.
\begin{remark}
\label{decompose injective res}
Let $I$ be a homogeneous ideal of $R$ and $M\in \Mod_R$. Set $X:=\Proj(R)$ and $\widetilde{M}$ to be the sheaf on $X$ induced by $M$. Set $Z:=\Proj(R/I)\subseteq X$ and $U:=X\backslash Z$.

It follows from \cite[2.1.5]{EGAIII1} ({\it cf.} \cite[\S 20.3]{BrodmannSharp}) that there is an exact sequence in $\Mod_R$
\[0\to H^0_I(M)\to M\to \bigoplus_{\ell\in \ZZ}H^0(U,\widetilde{M}(\ell))\to H^1_I(M)\to 0\]
and that there are isomorphisms in $\Mod_R$
\[ \bigoplus_{\ell\in \ZZ}H^i(U,\widetilde{M}(\ell))\cong H^{i+1}_I(M)\quad \forall i\geq 1.\]

Let $E^{\bullet}$ be an injective resolution of $M$ in $\Mod_R$. For each $E^i$, we have a decomposition $E^i\cong \bigoplus_{\alpha}\E(R/P_{\alpha})(\ell_{\alpha})$.  Set $\leftindex^{\prime}{E}{^\bullet}:=\Gamma_I(E^{\bullet})$. Then $\leftindex^{\prime}{E}{^i}=\Gamma_I(E^i)\cong\bigoplus_{I\subseteq P_\alpha}\E(R/P_{\alpha})(\ell_{\alpha})$ by Remark \ref{structure of injectives}. Hence $\leftindex^{\prime}{E}{^\bullet}$ is a subcomplex of $E^{\bullet}$. Then set $\leftindex^{\prime\prime}{E}{^\bullet}:= E^{\bullet}/\leftindex^{\prime}{E}^{\bullet}$. We have a short exact sequence of complexes: 
\[0\to \leftindex^{\prime}{E}{^\bullet}\to E^{\bullet} \to \leftindex^{\prime\prime}{E}{^\bullet}\to 0.\] 
Note that $H^i(\leftindex^{\prime}{E}{^\bullet})\cong H^i_I(M)$ and that $H^i(E^{\bullet})=\begin{cases}M&i=0\\ 0&i\neq 0 \end{cases}$. Then it follows that
%\begin{align*}
\[H^0(\leftindex^{\prime\prime}{E}{^\bullet})\cong \bigoplus_{\ell\in \ZZ}H^0(U,\widetilde{M}(\ell))\ {\rm and\ }  
H^i(\leftindex^{\prime\prime}{E}{^\bullet})\cong  H^{i+1}_I(M)\cong \bigoplus_{\ell\in \ZZ}H^i(U,\widetilde{M}(\ell))\ {\rm for\ } i\geq 1
\]
%\end{align*}
\end{remark}

Since we will need to consider the module $\bigoplus_{\ell\in \ZZ}H^0(U,\widetilde{M}(\ell))$ in the next sections, we introduce the following definition.
\begin{defn}
\label{defn: sat}
Let $R$ be a standard graded ring and $I$ be a homogeneous ideal. Set $X:=\Proj(R)$, $Z:=\Proj(R/I)\subseteq X$ and $U:=X\backslash Z$. For each $M\in \Mod_R$, we define the \emph{$I$-saturation of $M$} as
\[M^{sat}_I:=\bigoplus_{\ell\in \ZZ}H^0(U,\widetilde{M}(\ell))\]
where $\widetilde{M}$ is the sheaf on $X$ induced by $M$.

In particular, when $I=\frakm$ (the homogeneous maximal ideal of $R$), $U$ coincides with $X$ and we define
\[M^{sat}:=\bigoplus_{\ell\in \ZZ}H^0(X,\widetilde{M}(\ell)).\]
\end{defn}

The following is well-known.
\begin{remark}
\label{torsion or iso imply vanishing}
Let $M$ be an object in $\Mod_R$. 
\begin{enumerate}
\item 
If $M=\Gamma_I(M)$, then $H^k_I(M) = 0$ for $k>0$.
\item
If there exists a homogeneous element $f \in I$ of degree $\ell$ such that $M(-\ell)\xrightarrow{f}M$ is an isomorphism, then $H^k_I(M) = 0$ for all $k$.
\end{enumerate}
\end{remark}

Next we recall the definition of the Segr\'{e} product of two standard graded rings $R$ and $S$ ({\it cf.} \cite[p.~811]{ChowUnmixedness}) and we will follow the notations in \cite{GotoWatanabeI}.
\begin{defn}
\label{defn graded and Segre prod}
Let $R,S$ be standard graded rings and let $M\in \Mod_R,N\in\Mod_S$. The \emph{Segr\'{e} product} of $M$ and $N$ is defined as
\[M\#N:=\bigoplus_{n\in \ZZ}M_n\otimes_{\kk} N_n.\]
\end{defn}

Some basic properties of $M\#N$ can be summarized as follows ({\it cf.} \cite[\S4]{GotoWatanabeI}).
\begin{remark}
\label{Segre is exact}
Let $R,S,M,N$ as in Definition \ref{defn graded and Segre prod}. Set $T:=R\#S$.
\begin{enumerate}
\item One can check that 
\[M\#-: \Mod_S\to \Mod_T\quad {\rm and}\quad -\#N: \Mod_R\to \Mod_T\]
are both exact functors and commute with direct sum; this is \cite[4.0.3]{GotoWatanabeI}.
\item Assume that $M\to M'$ is injective in $\Mod_R$ and $N\to N'$ is injective in $\Mod_S$. Then $M\#N\to M'\#N'$ is injective in $\Mod_{R\#S}$. This can be see as follows. The map $M\#N\to M'\#N'$ is the composition of
\[M\#N\to M'\#N\to M'\#N'\]
in which both maps are injective since $-\#N$ and $M'\#-$ are exact functors.
\end{enumerate}
\end{remark}

Let $I$ and $J$ be homogeneous ideals of $R$ and $S$, respectively. Set $\fraka := I \# J$. Let $M\in \Mod_R$ and $N\in \Mod_S$. The following is a straightforward consequence of Remark \ref{torsion or iso imply vanishing}.
\begin{prop}
\label{Segre preserve torsion and auto}
Let $M\in \Mod_R$ and $N\in \Mod_S$.
\begin{enumerate}
\item	\label{torsion imply vanishing} If $M$ is $I$-torsion, then $M \# N$ is $I \# J$-torsion. Consequently, in this case, $H^k_{I\#J}(M\#N)=0$ for all $k>0$.
\item \label{iso imply vanishing} If there exists a homogeneous element $f \in I$ and a homogenous element $g\in J$ such that $M(-\deg(f))\xrightarrow{f}M$ is an isomorphism in $\Mod_R$ and $N(-\deg(g))\xrightarrow{g}N$ is an isomorphism in $\Mod_S$, then the multiplication
\[(M\#N)(-(\deg(f)\deg(g)))\xrightarrow{f^{\deg(g)}\#g^{\deg(f)}}M\#N\]
is an isomorphism in $\Mod_T$. Consequently, in this case, $H^k_{I\#J}(M\#N)=0$ for all $k$.
\end{enumerate}
\end{prop}

We recall the following well-known fact in homological algebra ({\it cf.} \cite[p.~205]{HartshorneAG}).
\begin{remark}
\label{gamma res computes derived}
Let $\scr{A}$ be abelian category with enough injective objects and let $\Gamma:\scr{A}\to \scr{A}$ be an covariant left-exact functor. Let $M$ be an object in $\scr{A}$. A complex 
\[C^{\bullet}:0\to C^0\to C^1\to\cdots\] 
is called a $\Gamma$-acyclic resolution of $M$ if
\begin{enumerate}
 \item $C^{\bullet}$ is a resolution of $M$; that is, $H^0(C^{\bullet})\cong M$ and $H^j(C^{\bullet})=0$ for all $j\neq 0$; and 
 \item $\mathscr{R}^j\Gamma(C^i)=0$ for all $i\geq 0$ and $j\neq 0$, where $\mathscr{R}^j\Gamma$ is the $j$-th right derived functor of $\Gamma$.
 \end{enumerate}
 If $C^{\bullet}$ is a $\Gamma$-acyclic resolution of $M$, then
\[\mathscr{R}^j\Gamma(M)\cong H^j(\Gamma(C^{\bullet})).\]
 \end{remark}

\begin{remark}
In general, the Segr\'{e} product of injective modules may no longer be an injective module over the Segr\'{e} product of rings. For instance, let $R=\kk[x,u]$ and $S=\kk[u,v]$. Set $M:=\E_R(R)$ and $N:=\E_S(\kk)\cong H^2_{(u,v)}(S)(-2)$. Then \cite[Example~V.5.6]{Stuckrad_Vogel_Book} shows that $\Ext^1_{R\#S}(\kk,M\#N)\neq 0$ and consequently $M\#N$ is not an injective $R\#S$-module.
\end{remark}

However, if $E^{\bullet}$ is an injective resolution of $M$ in $\Mod_R$ and $F^{\bullet}$ is an injective resolution of $N$ in $\Mod_S$, then one can construct a complex $E^{\bullet}\#F^{\bullet}$ such that the complex $\Gamma_{I\#J}(E^{\bullet}\#F^{\bullet})$ calculates local cohomology modules $H^k_{I\#J}(M\#N)$. We explain this next. 

\begin{defn}
\label{defn: segre of complexes}
Let $(A^{\bullet}, d^{\bullet}_A)$ be a complex in $\Mod_R$ and $(B^{\bullet}, d^{\bullet}_B)$ be a complex in $\Mod_S$. We construct a complex $(C^{\bullet},d^{\bullet}_C)$ as follows:
\begin{itemize}
\item $C^k := \bigoplus_{i+j=k} A^i \# B^j$ for each integer $k$;
\item $d^k_C = \sum_{i+j=k} d^i_A \# \id +(-1)^i (\id \# d^j_B))$, where $\id$ denotes the identity map.
\end{itemize}
We will denote this complex $(C^{\bullet},d^{\bullet}_C)$ by $A^{\bullet}\#B^{\bullet}$ and call it the {\it Segr\'{e} product of $A^{\bullet}$ and $B^{\bullet}$}.
\end{defn}

The following is \cite[Lemma 4.1.4]{GotoWatanabeI}:
\begin{lemma}
\label{cohomology of Segre prod of complex}
Let $A^{\bullet}, B^{\bullet},A^{\bullet}\#B^{\bullet}$ be complexes be as above. Then
\[H^k(A^{\bullet}\#B^{\bullet})\cong \bigoplus_{i+j=k}H^i(A^{\bullet})\#H^j(B^{\bullet}).\]
\end{lemma}

The following theorem shows that local cohomology supported in the Segr\'{e} product of ideals can be computed using the Segr\'{e} product of injective resolutions.
\begin{theorem}
\label{Segre product produce a Gamma resolution}
Let $M$ be an object in $\Mod_R$ and $N$ be an object in $\Mod_S$. Let $0\to M\to E^{\bullet}$ (and $0\to N\to F^{\bullet}$) be an injective resolution in $\Mod_R$ (in $\Mod_S$, respectively). Then, for every homogeneous ideal $I$ of $R$ and every homogeneous ideal $J$ of $S$, the complex $E^{\bullet}\#F^{\bullet}$ is a $\Gamma_{I\#J}$-acyclic resolution of $M\# N$ in $\Mod_{R\#S}$. 

In particular, $H^k(\Gamma_{I\#J}(E^{\bullet}\#F^{\bullet}))\cong H^k_{I\#J}(M\#N)$ for every $k$.
\end{theorem}
\begin{proof}
Denote $E^{\bullet}\#F^{\bullet}$ by $C^{\bullet}$. According to Remark \ref{gamma res computes derived}, it suffices to check $C^{\bullet}$ is a resolution of $M\#N$ and $H^t_{I\#J}(C^k)=0$ for all $t\geq 1$ and all $k$.

Since $H^i_I(E^{\bullet})=0$ for $i\geq 1$ and $H^j_J(F^{\bullet})=0$ for $j\geq 1$, it follows from Lemma \ref{cohomology of Segre prod of complex} that $C^{\bullet}$ is a resolution of $M\#N$. It remains to show that $H^t_{I\#J}(C^k)=0$ for all $t\geq 1$ and all $k$.

It follows from Remark \ref{structure of injectives} that we are reduced to proving that 
\[H^t_{I\#J}(\E(R/P)\#\E(S/Q))=0\] 
for all $t\geq 1$, all $P\in \hSpec(R)$, and all $Q\in \hSpec(S)$. If $I\subseteq P$ or $J\subseteq Q$, then $\E(R/P)$ is $I$-torsion or $\E(S/Q)$ is $J$-torsion. Thus, $\E(R/P)\#\E(S/Q))$ is $I\#J$-torsion and hence $H^t_{I\#J}(\E(R/P)\#\E(S/Q))=0$ for all $t\geq 1$ according to Proposition \ref{Segre preserve torsion and auto}(\ref{torsion imply vanishing}). Otherwise, there is a homogeneous element $f\in I$ not in $P$ and a homogeneous element $g\in J$ not in $Q$. Then it follows from Proposition \ref{Segre preserve torsion and auto}(\ref{iso imply vanishing}) that $H^t_{I\#J}(\E(R/P)\#\E(S/Q))=0$ for all $t\geq 1$. This completes the proof of our theorem.
\end{proof}

%%%%%%%%%%%%%%%%%%%%%%%%
\section{A K\"{u}nneth formula for local cohomology}
\label{Kunneth formula}
%%%%%%%%%%%%%%%%%%%%%%%%
Before we proceed to our main theorems, we would like to fix some notations. 

\begin{notation}
Let $R$ and $S$ be standard graded rings. 
\begin{itemize}
\item For each homogeneous ideal $I$ in $R$, we will denote the complement of $\Proj(R/I)$ in $\Proj(R)$ by $U$; likewise, for each homogeneous ideal $J$ in $S$, we will denote the complement of $\Proj(S/J)$ in $\Proj(S)$ by $V$. 
\item We will denote $R\#S$ by $T$ and $I\#J$ by $\fraka$. 
\item Denote the homogeneous maximal ideal of $R$, $S$, and $T$ by $\frakm_R$, $\frakm_S$, and $\frakm_T$, respectively. Note that $\frakm_T=\frakm_R\# \frakm_S$.   
\item For each $M\in \Mod_R$, we will denote by $\widetilde{M}$ the sheaf induced by $M$ on $\Proj(R)$; likewise we will denote by $\widetilde{N}$ the sheaf induced by $M$ on $\Proj(S)$. 
\item Recall $M^{sat}_I:=\bigoplus_{\ell\in \ZZ}H^0(U,\widetilde{M}(\ell))$ is the $I$-saturation of $M$ and $N^{sat}_J:=\bigoplus_{\ell\in \ZZ}H^0(V,\widetilde{N}(\ell))$ is the $J$-saturation of $N$. 
\item When $I=\frakm_R$, we denote by $M^{sat}$ the $\frakm_R$-saturation of $M$.
\end{itemize} 
\end{notation}

With these notations, we can state our main theorem:

\begin{theorem}
\label{Kunneth formula: two components}
Given the notations as above, we have an exact sequence
\begin{equation}
\label{thm: exact sequence}
0\to H^0_{\fraka}(M\#N)\to M\#N\to M^{sat}_I\#N^{sat}_J\to H^1_{\fraka}(M\#N)\to 0
\end{equation}
and isomorphisms
\begin{equation}
\label{thm: iso}
H^k_{\fraka}(M\#N)\cong \Big( M^{sat}_I\#H^k_J(N)\Big)\oplus \Big(H^k_I(M)\#N^{sat}_J\Big)\oplus \Big(\bigoplus_{i+j=k+1}H^i_I(M)\#H^j_J(N)\Big)
\end{equation}
for all $k\geq 2$.
\end{theorem}
\begin{proof}
Let $E^{\bullet}$ be an injective resolution of $M$ in $\Mod_R$ and $F^{\bullet}$ be an injective resolution of $N$ in $\Mod_S$. Set $C^{\bullet}:=E^{\bullet}\#F^{\bullet}$ as in Definition \ref{defn: segre of complexes}. It follows from Theorem \ref{Segre product produce a Gamma resolution} that
\[H^k_{\fraka}(M\#N)\cong H^k(\Gamma_{\fraka}(C^{\bullet})),\ \forall\ k\geq 0.\]
Set $\leftindex^{\prime}{E}{^\bullet}:=\Gamma_I(E^{\bullet})$ (which is a subcomplex of $E^{\bullet}$) and $\leftindex^{\prime\prime}{E}^{\bullet}:=E^{\bullet}/\Gamma_I(E^{\bullet})$. We define $\leftindex^{\prime}{F}{^\bullet}$ and $\leftindex^{\prime\prime}{F}{^\bullet}$ analogously. Following Remark \ref{decompose injective res}, we analyze modules in $\leftindex^{\prime}{E}{^\bullet}$ and $\leftindex^{\prime\prime}{E}{^\bullet}$ by decomposing each $E^i\cong \leftindex^{\prime}{E}{^i}\bigoplus \leftindex^{\prime\prime}{E}^{i}$ with respect to $I$; that is, $I$ is contained in every prime ideal appearing in $\leftindex^{\prime}{E}{^i}$ and not contained in any prime ideal appearing in $\leftindex^{\prime\prime}{E}^{i}$; likewise we decompose each $F^i\cong \leftindex^{\prime}{F}{^i}\bigoplus \leftindex^{\prime\prime}{F}^{i}$ with respect to $J$. Then the $i$-th terms in $\leftindex^{\prime}{E}^{\bullet}$, $\leftindex^{\prime\prime}{E}^{\bullet}$, $\leftindex^{\prime}{F}^{\bullet}$ and $\leftindex^{\prime\prime}{F}^{\bullet}$ are precisely $\leftindex^{\prime}{E}^{i}$, $\leftindex^{\prime\prime}{E}^{i}$, $\leftindex^{\prime}{F}^{i}$ and $\leftindex^{\prime\prime}{F}^{i}$, respectively. Set 
\[\leftindex^{\prime}{C}^{k}:=\bigoplus_{i+j=k}\Big( (\leftindex^{\prime}{E}^{i}\#\leftindex^{\prime}{F}^{j})\oplus(\leftindex^{\prime}{E}^{i}\# \leftindex^{\prime\prime}{F}^{j})\oplus(\leftindex^{\prime\prime}{E}^{i}\#\leftindex^{\prime}{F}^{j})\Big)\ {\rm and\ }\leftindex^{\prime\prime}{C}^{k}:=\bigoplus_{i+j=k}\Big(\leftindex^{\prime\prime}{E}^{i}\#\leftindex^{\prime\prime}{F}^{j}\Big).\]
Then one can check $\leftindex^{\prime}{C}^{\bullet}= \Gamma_{\fraka}(C^{\bullet})$ which is a subcomplex of $C^{\bullet}$; consequently we have an exact sequence of complexes

\[0\to \leftindex^{\prime}{C}{^\bullet}\to C^{\bullet} \to \leftindex^{\prime\prime}{C}{^\bullet}\to 0\]

which induces an exact sequence 
\begin{equation}
\label{exact sequence for C}
0\to H^0(\leftindex^{\prime}{C}{^\bullet})\to H^0(C^{\bullet}) \to H^0(\leftindex^{\prime\prime}{C}{^\bullet})\to H^1(\leftindex^{\prime}{C}{^\bullet})\to 0
\end{equation}
and isomorphisms
\begin{equation}
\label{iso for C}
H^k(\leftindex^{\prime}{C}{^\bullet})\cong H^{k-1}(\leftindex^{\prime\prime}{C}{^\bullet})\quad \forall k\geq 2
\end{equation}
since $H^k(C^{\bullet})=0$ for all $k\geq 1$.

It remains to treat cohomology groups of $\leftindex^{\prime\prime}{C}{^\bullet}$. Note that 
\[\leftindex^{\prime\prime}{C}{^\bullet}=\leftindex^{\prime\prime}{E}{^\bullet}\#\leftindex^{\prime\prime}{F}{^\bullet}.\]
It follows from Lemma \ref{cohomology of Segre prod of complex} that
\[H^{t}(\leftindex^{\prime\prime}{C}{^\bullet})\cong \bigoplus_{i+j=t}\Big( H^i(\leftindex^{\prime\prime}{E}{^\bullet})\# H^j(\leftindex^{\prime\prime}{F}{^\bullet})\Big)\quad \forall t\geq 0.\]
Now it follows from Remark \ref{decompose injective res} that 
\begin{align*}
H^0(\leftindex^{\prime\prime}{E}{^\bullet})\cong M^{sat}_I\quad &{\rm and}\quad H^i(\leftindex^{\prime\prime}{E}{^\bullet})\cong H^{i+1}_I(M)\ i\geq 1\\
H^0(\leftindex^{\prime\prime}{F}{^\bullet})\cong N^{sat}_J\quad &{\rm and}\quad H^j(\leftindex^{\prime\prime}{F}{^\bullet})\cong H^{j+1}_J(N)\ j\geq 1
\end{align*}
Consequently the exact sequence (\ref{thm: exact sequence}) follows from (\ref{exact sequence for C}) since $H^0(C{^\bullet})\cong M\#N$ and the isomorphisms (\ref{thm: iso}) follow from the isomorphisms (\ref{iso for C}). This completes the proof of our theorem.
\end{proof}

Specializing to the case $I=\frakm_R$ and $J=\frakm_S$ produces the following immediate corollary. 
\begin{corollary}
\label{kunneth for max ideals}
Given the notations as above, we have an exact sequence
\begin{equation}
\label{thm: exact sequence max ideals}
0\to H^0_{\frakm_T}(M\#N)\to M\#N\to M^{sat}\#N^{sat}\to H^1_{\frakm_T}(M\#N)\to 0
\end{equation}
and isomorphisms
\begin{equation}
\label{thm: iso max ideals}
H^k_{\frakm_T}(M\#N)\cong \Big( M^{sat}\#H^k_{\frakm_S}(N)\Big)\oplus \Big(H^k_{\frakm_R}(M)\#N^{sat}\Big)\oplus \Big(\bigoplus_{i+j=k+1}H^i_{\frakm_R}(M)\#H^j_{\frakm_S}(N)\Big)
\end{equation}
for all $k\geq 2$.
\end{corollary}

\begin{remark}
By a straightforward induction on the number of graded rings, one can deduce a K\"{u}nneth formula for local cohomology of the Segr\'{e} product of any finite number of graded rings.
\end{remark}

Our Theorem \ref{Kunneth formula: two components} also produces the following consequence on the behavior of saturations of modules under Segr\'{e} products.
\begin{corollary}
Given the notations as above,
\[(M\#N)^{sat}_{I\#J}\cong M^{sat}_I\#N^{sat}_J.\]
That is, if we denote by $W$ the complement of the closed subset defined by $I\#J$ in $\Proj(R\#S)$, then
\[\bigoplus_{\ell\in \ZZ}H^0(W,\widetilde{M\#N}(\ell))\cong (\bigoplus_{\ell\in \ZZ}H^0(U,\widetilde{M}(\ell)))\#(\bigoplus_{\ell\in \ZZ}H^0(V,\widetilde{N}(\ell)))\]
\end{corollary}

%%%%%%%%%%%%%%%%%%%%%%%%%%%%%
\section{A sharp lower bound on depth}
\label{depth section}
%%%%%%%%%%%%%%%%%%%%%%%%%%%%%

In this section, we consider some applications of our K\"{u}nneth formula in Theorem \ref{Kunneth formula: two components}.

We begin with applications to depth. Recall that the depth\footnote{This is called the {\it grade} by some authors.} of a commutative ring $A$ in an ideal $\fraka$ is defined as
\[\Depth_{\fraka}(A):=\inf\{i\mid \Ext^i_A(A/\fraka,A)\neq 0\},\]
which is also the length of a maximal regular sequence in $\fraka$. When $A$ is a standard graded ring and $\frakm$ is the homogeneous maximal ideal, we set $\Depth(A):=\Depth_{\frakm}(A)$.

\begin{theorem}
\label{thm: lower bound depth}
Let $R,S$ be standard graded rings and $I,J$ be homogeneous ideals in $R,S$, respectively. Then
\[\Depth_{I\#J}(R\#S)\geq \min\{\Depth_I(R),\Depth_J(S)\}.\]
\end{theorem}
\begin{proof}
Set $m=\min\{\Depth_I(R),\Depth_J(S)\}$. It suffices to show that $H^k_{I\#J}(R\#S)=0$ for all $k\leq m-1$.

When $m=0$, there is nothing to prove. When $m=1$, we have $H^0_I(R)=H^0_J(S)=0$. It follows frome Remark \ref{decompose injective res} that both $R\to R^{sat}_I$ and $S\to S^{sat}_J$ are injective. Hence Remark \ref{Segre is exact}(2) implies that $R\#S\to R^{sat}_I\#S^{sat}_J$ is also injective. Thus, by the exact sequence (\ref{thm: exact sequence}) in Theorem \ref{Kunneth formula: two components} that $H^0_{I\#J}(R\#S)=0$. This proves that case when $m=1$.

When $m=2$, we have $H^i_I(R)=H^i_J(S)=0$ for $i=0,1$. It follows that both $R\cong R^{sat}_I$ and $S\cong S^{sat}_J$. Hence $R\#S\cong R^{sat}_I\#S^{sat}_J$. Thus, by the exact sequence (\ref{thm: exact sequence}) in Theorem \ref{Kunneth formula: two components} that $H^0_{I\#J}(R\#S)=H^1_{I\#J}(R\#S)=0$. This proves that case when $m=2$.

Assume that $m\geq 3$ and $k\leq m-1$. It follows from the previous paragraph that $H^0_{I\#J}(R\#S)=H^1_{I\#J}(R\#S)=0$. Hence we may assume that $k\geq 2$. Since $k\leq m-1$, it is clear that $H^k_I(R)=H^k_J(S)=0$. Given any pair $(i,j)$ with $i+j=k+1$, one can check either $H^i_I(R)=0$ or $H^j_J(S)=0$. Therefore $H^k_{I\#J}(R\#S)=0$ by the isomorphisms (\ref{thm: iso}) in Theorem \ref{Kunneth formula: two components}. This completes the proof.
\end{proof}

As we will see, the lower bound in Theorem \ref{thm: lower bound depth} is optimal. In order to explain this, we need the following observation on local cohomology.

\begin{prop}
\label{elements in positive degrees}
Let $R$ be a standard graded ring of dimension $d>0$ and $I$ be a homogeneous ideal of height $h<d$. Then there exists an integer $\ell_0$ such that
\[H^h_I(R)_{\ell}\neq 0\quad \forall \ell\geq \ell_0.\] 
\end{prop}
\begin{proof}
Let $\frakp$ be a minimal prime of $I$ of height $h$; note that $\frakp$ is necessarily homogeneous. Since $\sqrt{IR_\frakp}=\frakp R_\frakp$ and $h=\dim(R_\frakp)$, by the Grothendieck nonvanishing theorem (\cite[6.1.4]{BrodmannSharp}), $H^h_I(R)_\frakp=H^h_{IR_\frakp}(R_{\frakp})\neq 0$. It follows that $\frakp$ is in the support of $H^h_I(R)$. Since $\hgt(\frakp)=\hgt(I)$, it follows that $\frakp$ is a minimal element in the support of $H^h_I(R)$; hence $\frakp$ is an associated prime of $H^h_I(R)$. Therefore, there is a homogeneous element $z\in H^h_I(R)$ such that the $R$-module homomorphism $R/\frakp\xrightarrow{1\mapsto z} H^h_I(R)$ is injective. Set $\ell_0:=\deg(z)$. By the assumptions, $\dim(R/\frakp)>0$ and consequently $H^h_I(R)_{\ell}\neq 0$ for all $\ell\geq \ell_0$. 
\end{proof}

\begin{theorem}
\label{depth Segre}
Let $R,S$ be standard graded rings and $I,J$ be positive dimensional homogeneous ideals in $R,S$, respectively. Assume that $R$ and $S$ are Cohen-Macaulay. Then
\[\Depth_{I\#J}(R\#S)=\min\{\Depth_I(R),\Depth_J(S)\}.\]
\end{theorem}
\begin{proof}
Since $R$ and $S$ are Cohen-Macaulay, $\Depth_I(R)=\hgt_R(I)$ and $\Depth_J(S)=\hgt_S(J)$. 

Set $m=\min\{\Depth_I(R),\Depth_J(S)\}$, $T:=R\#S$, and $\fraka:=I\#J$. Without loss of generality, we assume that $m=\Depth_I(R)=\hgt_R(I)$. It follows from Proposition \ref{elements in positive degrees} that $H^m_I(R)_j\neq 0$ for $j\gg 0$.

By Theorem \ref{thm: lower bound depth}, it suffices to show that $H^m_{\fraka}(T)\neq 0$.

First, we assume that $m\geq 2$; consequently $R\cong R^{sat}_I$ and $S\cong S^{sat}_J$. In this case, it follows from Theorem \ref{Kunneth formula: two components} that
\[H^m_{\fraka}(T)\cong \Big( R\#H^m_J(S)\Big)\oplus \Big(H^m_I(R)\# S\Big)\oplus \Big(\bigoplus_{i+j=m+1}H^i_I(R)\#H^j_J(S)\Big)
\]
It follows from Proposition \ref{elements in positive degrees} that the direct summand $H^m_I(R)\# S\neq 0$. Therefore, $H^m_{\fraka}(T)\neq 0$ and hence $\Depth_{\fraka}(T)= m$. This proves the case when $m\geq 2$.

Next we assume that $m=1$. In this case, $H^1_I(R)_j\neq 0$ for $j\gg 0$ by Proposition \ref{elements in positive degrees} and there is a short exact sequence $0\to R\to R^{sat}_I\to H^1_I(R)\to 0$ in $\Mod_R$. By Remark \ref{Segre is exact}, there is an exact sequence $0\to R\#S^{sat}_J\to R^{sat}_I\#S^{sat}_J\to H^1_I(R)\#S^{sat}_J\to 0$. Consequently, the map $R\#S^{sat}_J\to R^{sat}_I\#S^{sat}_J$ is not surjective as $H^1_I(R)\#S^{sat}_J\neq 0$. It follows that $R\#S\to R^{sat}_I\#S^{sat}_J$ is not surjective. Therefore $H^1_{\fraka}(T)\neq 0$. This proves that case when $m=1$.

Finally, we assume $m=0$; that is, $\Depth_I(R)=0$. Then, $H^0_I(R)\neq 0$. Pick a nonzero homogeneous $f\in R$ such that $f$ is $I$-power torsion. Then for every nonzero $g\in S$ such that $\deg(f)=\deg(g)$, the element $f\#g$ is $I\#J$-power torsion. This shows that $H^0_{\fraka}(T)\neq 0$. Hence $\Depth_{\fraka}(T)=0=m$. This proves the case when $m=0$ and hence the proof of our theorem.
\end{proof}

%%%%%%%%%%%%%%%%%%%%%%%%%%%%%%%%%%%%
\section{Results on graded $\calD$-modules with an application to cohomological dimension}
\label{graded D-module and cd section}
%%%%%%%%%%%%%%%%%%%%%%%%%%%%%%%%%%%%
Let $R=\kk[x_1,\dots,x_d]$ be a polynomial ring over a field $\kk$ and let $\frakm$ denote the maximal irrelevant ideal $(x_1,\dots,x_d)$. Let $\mathcal{D}(R;\kk)$ denote the ring of $\kk$-linear differential operators on $R$. Then $\mathcal{D}(R;\kk)=R\langle \partial_i^{[t]}\mid t\in \mathbb{N}, 1\leq i\leq d\rangle$ where $\partial_i^{[t]}:=\frac{1}{t!}\frac{\partial^t}{\partial x^t_i}$. The (noncommutative) ring $\mathcal{D}(R;\kk)$ admits a natural grading: $\deg(x_i)=1$ and $\deg(\partial^{[t]}_i)=-t$ for all $x_i$ and all $t\geq 1$. When $R$ and $\kk$ are clear from the context, we will denote $\mathcal{D}(R;\kk)$ be $\scr{D}$. By a graded $\calD$-module we mean a $\ZZ$-graded left $\calD$-module; primary examples of graded $\calD$-modules are local cohomology modules $H^j_I(R)$ for homogeneous ideals $I$ of $R$.

\begin{remark}
The usual Leibniz rule can be extended to $\partial_i^{[t]}$ as follows ({\it cf.} \cite[Corollary~2.2]{LyubeznikBasicResult}):
\begin{equation}
\label{general Leibniz}
\partial_i^{[t]}x^s_i=\sum_{j=0}^tx^{s-j}_i\partial_i^{[t-j]}
\end{equation}
\end{remark}

\begin{lemma}
\label{lem: derivative inseparable}
Let $R=\kk[x_1,\dots,x_d]$ be a polynomial ring over a field $\kk$. Assume that $\Char(\kk)=p>0$. If $f,g\in R$ satisfy $f(x_1,\dots,x_d)=g(x^{p^e}_1,\dots,x^{p^e}_d)$ for an integer $e\geq 0$, then 
\[(\partial_i^{[p^e]}f)(x_1,\dots,x_d)=(\frac{\partial g}{\partial x_i})(x^{p^e}_1,\dots,x^{p^e}_d)\] 
for each $x_i$.
\end{lemma}
\begin{proof}
Since $\partial_i^{[p^e]}$ is additive, we reduced to the case when $f$ is a monomial. As $\partial_i^{[p^e]}$ commutes with $x_j$ ($j\neq i$), we are reduced to the case when $f=x^{ap^e}_i$; hence $g=x^a_i$. It follows from (\ref{general Leibniz}) that
\[\partial_i^{[p^e]}x^{ap^e}_i=\binom{ap^e}{p^e}x^{(a-1)p^e}_i.\]
By Lucas' Theorem $\binom{ap^e}{p^e}\equiv a$ (mod $p$). Our results follows.
\end{proof}

The following is the main technical result of this section.
\begin{theorem}
\label{thm: nonvanishing negative degree}
Let $R=\kk[x_1,\dots,x_d]$ be a polynomial ring over a field $\kk$  and let $M$ be a nonzero graded $\calD$-module.
\begin{enumerate}
\item 
\label{max ideal being associated}
If $\frakm$ is an associated prime of $M$, then there is an integer $\ell$ such that $M_{\ell'}\neq 0$ for all $\ell'\leq \ell$.
\item 
\label{zero in one degree above}
If there is an integer $\ell$ such that $M_{\ell}\neq 0$ and $M_{\ell+1}=0$, then $M_{\ell'}\neq 0$ for all $\ell'\leq \ell$.
\item 
\label{never zero}
If $\frakm$ is not an associated prime of $M$ and each element of $M$ is annihilated by a nonzero polynomial, then \[M_{\ell}\neq 0,\quad \forall \ell\in \ZZ.\]
\end{enumerate}
\end{theorem}
\begin{proof}
\begin{enumerate}
\item Since $\frakm$ is an associated prime, there is a nonzero homogeneous element $z\in M$ such that $\frakm z=0$. Set $\ell=\deg(z)$. Since $M$ is a $\calD$-module, there is a $\calD$-linear map
\[\calD/\frakm \calD\xrightarrow{1\mapsto z}M.\]
As $\frakm \calD$ is a maximal ideal in $\calD$, this maps must be injective. Since 
\[\calD/\frakm \calD\cong \kk[\partial_i^{[t]}\mid t\in \mathbb{N}, 1\leq i\leq d],\] 
our result follows since $\deg(\partial_i^{[t]})=-t$.
\item Since $M_{\ell}\neq 0$, there is a nonzero element $z\in M_{\ell}$. Since $M_{\ell+1}=0$, it follows that $\frakm =0$. Hence the result follows from (\ref{max ideal being associated}).

\item Assume otherwise; that is, there is an integer $\ell$ such that $M_{\ell}=0$. If $M_{\ell-1}\neq 0$, then each nonzero element in $M_{\ell-1}$ is annihilated by $\frakm$, a contradiction to the assumption that $\frakm$ is not an associated prime of $M$. Consequently, $M_{\ell'}=0$ for all $\ell'\leq \ell$. Since $M\neq 0$, there exists an integer $\ell_0$ such that $M_{\ell_0}\neq 0$ and $M_{\ell'}=0$ for all $\ell'<\ell_0$. 

Let $z\in M_{\ell_0}$ be a nonzero element. By our assumption, there is a nonzero polynomial $f\in R$ such that $fz=0$. We may assume that $f$ is homogeneous of the least degree among all polynomials that annihilate $z$. Note that $\deg(f)>0$ since $z\neq 0$. For each variable $x_i$, we have
\[\frac{\partial f}{\partial x_i}z=[\frac{\partial }{\partial x_i},f]z=\frac{\partial }{\partial x_i}(fz)-f(\frac{\partial }{\partial x_i}z)=0\]
where $\frac{\partial }{\partial x_i}z=0$ since $\deg(\frac{\partial }{\partial x_i}z)<\deg(z)=\ell_0$. Consequently $\frac{\partial f}{\partial x_i}z=0$ for each variable $x_i$. 

For the rest of the proof, we consider two cases: when $\Char{\kk}=0$ and when $\Char{\kk}=p>0$.

Assume that $\Char{\kk}=0$. Since $\deg(f)>0$, there must be at least one variable $x_i$ such that $\frac{\partial f}{\partial x_i}\neq 0$ (this is where we use the assumption that $\Char{\kk}=0$). But this contradicts to the assumption that $f$ has the least degree among all polynomials that annihilate $z$ and hence finishes the proof in the case when $\Char{\kk}=0$.

Assume $\Char{\kk}=p>0$. Since $\frac{\partial f}{\partial x_i}z=0$ for each variable $x_i$ and $\deg(\frac{\partial f}{\partial x_i})<\deg(f)$, we must have $\frac{\partial f}{\partial x_i}=0$ for each $x_i$. Consequently there is a polynomial $g$ such that 
\[f(x_1,\dots,x_d)=g(x^p_1,\dots,x^p_d).\]
It follows from (\ref{general Leibniz}) that 
\[(\partial^{[p]}_if)z=[\partial^{[p]},f]z=\partial^{[p]}_i(fz)-f(\partial^{[p]}_iz)=0\]
where $\partial^{[p]}_iz=0$ since $\deg(\partial^{[p]}_iz)<\deg(z)=\ell_0$. By our assumption on $f$, we must have $\partial^{[p]}_if=0$ for each $x_i$. Consequently, it follows from Lemma \ref{lem: derivative inseparable} that $\frac{\partial g}{\partial x_i}=0$, Hence there is a polynomial $g_1$ such that $g(x_1,\dots,x_d)=g_1(x^{p}_1,\dots, x^{p}_d)$. That is, 
\[f(x_1,\dots,x_d)=g(x^{p^2}_1,\dots,x^{p^2}_d).\]
Repeating this process, we can conclude that 
\[f(x_1,\dots,x_d)=g(x^{p^e}_1,\dots,x^{p^e}_d),\]
for each integer $e\geq 0$. Hence $\deg(f)\geq p^e$ for each $e\geq 0$, which is impossible. 

Therefore, in both cases, we derive a contradiction from the assumption that there is an integer $\ell$ such that $M_{\ell}=0$. This finishes the proof of (\ref{never zero}).
\end{enumerate}
\end{proof}

We would like to apply Theorem \ref{thm: nonvanishing negative degree} to Eulerian graded $\scr{D}$-modules. Recall that an Eulerian graded $\scr{D}$-module is a graded $\scr{D}$-module $M$ such that
\[(\sum_{t_1+\cdots +t_d=t;t_1,\dots,t_d\geq 0}x^{t_1}_1\cdots x^{t_d}_d\partial^{[t_1]}_1\cdots \partial^{[t_d]}_d)z=\binom{\deg(z)}{t}z\]
for every homogeneous element $z\in M$ and every positive integer $t$.

\begin{theorem}
\label{nonzero in negative degree}
Let $R=\kk[x_1,\dots,x_d]$ be a polynomial ring over a field $\kk$  and let $M$ be a nonzero Eulerian graded $\scr{D}$-module. 
\begin{enumerate}
\item 
\label{supp dim 0}
If $\dim(\Supp_R(M))=0$, then 
\begin{enumerate}
\item $M_\ell=0$ for each integer $\ell>-d$, and
\item $M_{\ell}\neq 0$ for each integer $\ell\leq -d$.
\end{enumerate}
\item 
\label{supp dim p>0}
If $\dim(\Supp_R(M))>0$ and each element in $M$ is annihilated by a nonzero polynomial in $R$, then 
\[M_{\ell}\neq 0\quad \forall \ell\in \ZZ.\]
\end{enumerate}
In particular, if each element in $M$ is annihilated by a nonzero polynomial in $R$, then $M_{\ell}\neq 0$ for every integer $\ell\leq -d$.
\end{theorem}
\begin{proof}
If $\dim(\Supp_R(M))=0$, then it follows from \cite[Theorem~2.4]{LyubeznikDMod} and \cite[Theorem~5.6]{MaZhang} that $M$ is isomorphic to a direct sum of some copies of $H^d_{\frakm}(R)$ {\it and} this isomorphism preserves degree. Hence (\ref{supp dim 0}) follows.

When $\dim(\Supp_R(M))>0$, we consider $\overline{M}:=M/\Gamma_{\frakm}(M)$. It follows from \cite[Proposition 2.8]{MaZhang} that $\overline{M}$ is Eulerian graded. Since $\dim(\Supp_R(M))>0$, it follows that $\overline{M}\neq 0$. It is clear that $\frakm$ is not an associated prime of $\overline{M}$ and each element of $\overline{M}$ is also annihilated by a nonzero polynomial in $R$. Therefore, it  follows from \ref{thm: nonvanishing negative degree}(\ref{never zero}) that $\overline{M}_{\ell}\neq 0$ for all integers $\ell\in \ZZ$. Consequently $M_{\ell}\neq 0$ for all integers $\ell\in \ZZ$.
\end{proof}

Next, we would like to apply our Theorem \ref{Kunneth formula: two components} to the study of cohomological dimension whose definition we recall as follows. 

\begin{defn} 
\label{cd defn}
Let $A$ be a noetherian commutative ring and $\fraka$ be an ideal of $A$. The \emph{cohomological dimension} of the pair $(A,\fraka)$ is defined as
\[\cd_{\fraka}(A):=\sup\{i\in \ZZ \mid H^i_\fraka(A)\neq 0.\}\]
\end{defn}

\begin{theorem}
\label{upper bound on cd}
Let $R,S$ be standard graded rings and $I,J$ be homogeneous ideals in $R,S$, respectively.
\begin{enumerate}
\item If $\cd_I(R)\cd_J(S)=0$ (that is, $\cd_I(R)=0$ or $\cd_J(S)=0$), then 
\[\cd_{I\# J}(R\# S)\leq \cd_I(R)+\cd_J(S).\]
\item If $\cd_I(R)\cd_J(S)\neq 0$ (that is, $\cd_I(R)\geq 1$ and $\cd_J(S)\geq 1$), then 
\[\cd_{I\# J}(R\# S)\leq \cd_I(R)+\cd_J(S)-1.\]
\end{enumerate}
\end{theorem}
\begin{proof}
Set $r=\cd_I(R)$ and $s=\cd_J(S)$.
\begin{enumerate}
\item Assume $rs=0$. Without loss of generality, assume $r=0$. Then $r+s=s$. It suffices to show that $H^{k}_{I\#J}(R\#S)=0$ for all $k\geq s+1$. To this end, we will consider two cases: 
\begin{enumerate}
\item If $s=0$ as well, then $H^1_I(R)=H^1_J(S)=0$ and consequently $R\to R^{sat}_I$ and $S\to S^{sat}_J$ are surjective. It follows from Remark \ref{Segre is exact} that $R\#S\to R^{sat}_I\#S^{sat}_J$ is also surjective. Then it follows from the exact sequence (\ref{thm: exact sequence}) that $H^1_{I\#J}(R\#S)=0$. 

For each $k\geq 2$, we have 
\begin{enumerate}
\item $R^{sat}_I\#H^k_J(S)=0$ since $H^k_J(S)=0$;
\item $H^k_I(R)\# S^{sat}_J$ since $H^k_I(S)=0$; and
\item $\bigoplus_{i+j=k+1}H^i_I(R)\#H^j_J(S)=0$ since either $i\geq 1>r$ or $j\geq 1>s$.
\end{enumerate}
It now follows from the isomorphisms (\ref{thm: iso}) that $H^k_{I\#J}(R\#S)=0$. This proves the case when $r=s=0$.

\item If $s\geq 1$, then $s+1\geq 2$ and hence it follows from the isomorphisms (\ref{thm: iso}) that $H^k_{I\#J}(R\#S)=0$ for all $k\geq s+1\geq 2$. This proves the case when $r=0$ and $s\geq 1$.
\end{enumerate}
\item Assume that $r\geq 1$ and $s\geq 1$. It suffices to show that $H^{k}_{I\#J}(R\#S)=0$ for all $k\geq (r+s-1)+1=r+s\geq 2$. It is clear that $H^k_I(R)=H^k_J(S)=0$ for $k=r+s>r,s$. Since $k\geq 2$, it follows from the isomorphisms (\ref{thm: iso}) that it suffices to verify $H^i_I(R)\#H^j_J(S)=0$ for $i+j=k+1$. Since $i+j=k+1\geq r+s+1$, we have either $i>r$ or $j>s$; consequently either $H^i_I(R)=0$ or $H^j_J(S)=0$. This proves that $H^i_I(R)\#H^j_J(S)=0$ for $i+j=k+1$ and hence completes the proof of our theorem.
\end{enumerate}
\end{proof}

When $I=\frakm_R$ and $J=\frakm_S$, it is straightforward to check that $\cd_{I\#J}(R\#S)=\dim(R\#S)=\dim(R)+\dim(S)-1=\cd_I(R)+\cd_J(S)-1$. The following results show that the bound in Theorem \ref{upper bound on cd} is optimal for ideals of positive dimensions as well.

We state the result in the polynomial case first:
\begin{theorem}
\label{thm: optimal in polynomial case}
Let $R,S$ be polynomial rings over the same field $\kk$. Let $I,J$ be nonzero homogeneous ideals in $R,S$ respectively. Then
\[\cd_{I\# J}(R\# S)= \cd_I(R)+\cd_J(S)-1.\]
\end{theorem}
\begin{proof}
Since $H^{\cd_I(R)}_I(R)$ and $H^{\cd_J(S)}_J(S)$ are nonzero Eulerian graded (\cite[Proposition~5.2]{MaZhang}) and clearly each of their elements is annihilated by a nonzero polynomial, it follows from Theorem \ref{nonzero in negative degree} that 
\[H^{\cd_I(R)}_I(R)\#H^{\cd_J(S)}_J(S)\neq 0.\] 
Then our Theorem \ref{Kunneth formula: two components} and Theorem \ref{upper bound on cd} finish the proof.
\end{proof}

We end with the following result which shows that the bound in Theorem \ref{upper bound on cd} can be optimal even when the rings are not polynomial rings.
\begin{theorem}
\label{equality for cd}
Let $R,S$ be standard graded rings and $I,J$ be positive dimensional homogeneous ideals in $R,S$ respectively. Assume that 
\begin{enumerate}
\item both $I$ and $J$ are cohomological complete intersections; that is, $H^i_I(R)=0$ for $i\neq \hgt_R(I)$ and $H^j_J(S)=0$ for $j\neq \hgt_S(J)$; and 
\item $\hgt_R(I)+\hgt_S(J)\geq 3$.
\end{enumerate} 
Then
\[\cd_{I\# J}(R\# S)= \cd_I(R)+\cd_J(S)-1.\]
\end{theorem}
\begin{proof}
Set $k=\cd_I(R)+\cd_J(S)-1$.  Given Theorem \ref{upper bound on cd}, it suffices to show that $H^k_{I\#J}(R\#S)\neq 0$. 

It follows from our assumption that $\cd_I(R)=\hgt_R(I)$ and that $\cd_J(S)=\hgt_S(J)$. Since $\cd_I(R)+\cd_J(S)= \hgt_R(I)+\hgt_S(J)\geq 3$ by our assumption, $k\geq 2$. Therefore, by (\ref{thm: iso}), $H^k_{I\#J}(R\#S)$ contains a direct summand $H^{\hgt_R(I)}_I(R)\#H^{\hgt_S(J)}_J(S)$. It follows from Proposition \ref{elements in positive degrees}, $H^{\hgt_R(I)}_I(R)\#H^{\hgt_S(J)}_J(S)\neq 0$ which completes the proof.
\end{proof}


\begin{thebibliography}{GW78}

\bibitem[BSh13]{BrodmannSharp}
{\sc M.~Brodmann and R.~Sharp}: \emph{Local cohomology}, second ed., Cambridge
  Studies in Advanced Mathematics, vol. 136, Cambridge University Press,
  Cambridge, 2013.
  
\bibitem[BSch90]{BrunsSchwanzl}
{\sc W.~Bruns and R.~Schw\"{a}nzl}: \emph{The number of equations defining a determinantal variety}, Bull. London Math. Soc. \textbf{22} (1990), no. 5, 439--445.

\bibitem[BV88]{DetRingBook}
{\sc W.~Bruns and U.~Vetter}: \emph{Determinantal rings}, Lecture Notes in Mathematics, 1327. Springer-Verlag, Berlin, 1988. viii+236 pp.

\bibitem[Cho64]{ChowUnmixedness}
{\sc W.~L. Chow}: \emph{On unmixedness theorem}, Amer. J. Math. \textbf{86}
  (1964), 799--822. {\sf\scriptsize 171804}

\bibitem[GW78]{GotoWatanabeI}
{\sc S.~Goto and K.-I. Watanabe}: \emph{On graded rings, {I}}, J. Math. Soc.
  Japan \textbf{30} (1978), no.~2, 179--213.

\bibitem[GD61]{EGAIII1}
{\sc A.~Grothendieck and J.~Dieudonn\'{e}}: \emph{El\'{e}ments de
  g\'{e}om\'{e}trie alg\'{e}brique {III}: \'{E}tude cohomologique des faisceaux
  coh\'{e}rents, premi\`{e}re partie}, Publ. Math. IH\'{E}S \textbf{11} (1961),
  5--167.

\bibitem[HE71]{HochsterEagon}
{\sc M.~Hochster and J.~Eagon}: \emph{Cohen-Macaulay rings, invariant theory, and the generic perfection of determinantal loci}, Amer. J. Math. \textbf{93} (1971), 1020--1058. 


\bibitem[Har77]{HartshorneAG}
{\sc R.~Hartshorne}: \emph{Algebraic geometry}, Graduate Texts in Mathematics,
  vol.~52, Springer, New York, 1977.
  
  \bibitem[Lyu90]{LyubeznikDMod}
  {\sc G.~Lyubeznik}: \emph{Finiteness properties of local cohomology modules (an application of D-modules to commutative algebra)}, Invent. Math. \textbf{113} (1993), no. 1, 41--55.
  
  \bibitem[Lyu11]{LyubeznikBasicResult}
  {\sc G.~Lyubeznik}: \emph{A characteristic-free proof of a basic result on D-modules}, J. Pure Appl. Algebra \textbf{215} (2011), no. 8, 2019--2023.
  
  \bibitem[MZ14]{MaZhang}
{\sc L.~Ma and W.~Zhang}: \emph{Eulerian graded $\mathcal{D}$-modules}, Math. Res. Lett. \textbf{21} (2014), no. 1, 149--167.

\bibitem[PS73]{PeskineSzpiro}
{\sc C.~Peskine and L.~Szpiro}: \emph{Dimension projective finie et cohomologie
  locale. Applications \`{a} la d\'{e}monstration de conjectures de M. Auslander, H. Bass et A. Grothendieck.}, Publ. Math. {IH\'ES} \textbf{42} (1973), 47--119.
  
  \bibitem[Put22]{PuthenpurakalGradedComponents}
{\sc T.~J.~ Puthenpurakal}: \emph{Graded components of local cohomology modules}, Collect. Math. \textbf{73} (2022), no. 1, 135--171.

\bibitem[SV86]{Stuckrad_Vogel_Book}
{\sc J.~St\"{u}ckrad and W.~Vogel}: \emph{Buchsbaum rings and applications},
  Mathematische Monographien [Mathematical Monographs], vol.~21, VEB Deutscher
  Verlag der Wissenschaften, Berlin, 1986, An interaction between algebra,
  geometry, and topology. {\sf\scriptsize 873945}

\end{thebibliography}
\end{document}